\documentclass[a4paper,centertags,fleqn,reqno,12pt]{amsart}
\usepackage{amsfonts}
\usepackage[left=1.5cm, bottom=3cm]{geometry}
\usepackage{amssymb,amsmath,latexsym}
\usepackage{cite}
\usepackage{amsmath}
\usepackage[colorlinks=true,linkcolor=blue,citecolor=red]{hyperref}
\newtheorem{theorem}{Theorem}

\theoremstyle{definition}
\newtheorem{definition}{Definition}

\numberwithin{equation}{section}

\allowdisplaybreaks
\input{tcilatex}
\begin{document}
\title[Fekete-Szeg\"{o} inequality of bi-starlike and bi-convex functions]{Fekete-Szeg\"{o} inequality of bi-starlike and bi-convex functions of order $b$ associated with symmetric $q$-derivative in conic domains}
\author{R. B. Sharma$^{1}$,\; K. Rajya Laxmi$^{2}$ and N. Magesh$^{\ast,\; 3}$}
\address{
$^{1}$Department of Mathematics,
Kakatiya University, Warangal-506 009,
Telangana, India.
\texttt{e-mail:} rbsharma005@gmail.com}
\address{
$^{2}$Department of Mathematics,
SRIIT, Hyderabad, India.
\texttt{e-mail:} rajyalaxmi2206@gmail.com}
\address{$^{\ast,\; 3}$Post-Graduate and Research Department of Mathematics
Government Arts College for Men, Krishnagiri - 635 001, Tamilnadu, India.
\texttt{e-mail:} nmagi$\_$2000@yahoo.co.in}
\thanks{$^{\ast}$ Corresponding author}

\begin{abstract}
In this paper, two new subclasses of bi-univalent functions related to conic domains are defined by making use of  symmetric $q$-differential operator. The initial bounds for Fekete-Szeg\"{o} inequality for the functions $f$ in these classes are estimated.
\end{abstract}
\maketitle
\section{Introduction}
Let $ \mathfrak{\mathcal{A}} $  denotes the set of all functions which are analytic in the unit disc $ \Delta=\{z\in\mathbb{C}:|z|<1 \} $ with Taylor's series expansion of the form
\begin{align}
 f(z) = z + \sum \limits_{n=2}^{\infty} a_{n}z^{n}\label{1.1}
\end{align}
which are normalized by $f(0)=0, f'(0)=1$.\\
The subclass of $ \mathfrak{\mathcal{A}} $ consisting of all univalent functions is denoted by $ S $. That is
\begin{align}
S=\{f\in \mathfrak{\mathcal{A}}: f \text{ is univalent in }\Delta \}.\nonumber
\end{align}
A function $f\in \mathfrak{\mathcal{A}}$ is said to be a starlike function if
\begin{align}
\Re\Big(\dfrac{zf'(z)}{f(z)}\Big)>0,\;z\in\Delta.\nonumber
\end{align}
 A function $f\in \mathfrak{\mathcal{A}}$ is said to be a  convex function if
\begin{align}
\Re\Big(1+\dfrac{zf''(z)}{f'(z)}\Big)>0,\;z\in\Delta.\nonumber
\end{align}

\par Goodman \cite{good,goodc,goods} introduced the classes uniformly starlike and uniformly convex functions as subclasses of starlike and convex functions. A starlike function (or convex function) is said to be uniformly starlike (or uniformly convex) if the image of  every circular arc $\zeta$ contained in $\Delta$, with center at $\xi$ also in $\Delta$ is starlike (or convex) with respect to $f(\xi)$.   The class of uniformly starlike functions is represented by $UST$ and the class of uniformly convex functions is represented by $UCV$. The class of parabolic starlike functions is represented by $S_{p}.$ R\o nning \cite{ch350} and Ma-Minda \cite{ma,ma1}  independently gave the characterization for the classes $S_{p} $ and $UCV$ as follows.

\par A function $f\in \mathfrak{\mathcal{A}}$ is said to be in the class $S_{p}$ if and only if

\[\Re\Big(\dfrac{zf'(z)}{f(z)}\Big)>\Big|\dfrac{zf'(z)}{f(z)}-1\Big|,\;z\in \Delta.\nonumber\]

A function $f\in \mathfrak{\mathcal{A}}$ is said to be in the class $UCV$ if and only if
\[
\Re\left(1+\dfrac{zf''(z)}{f'(z)}\right)>\Big|\dfrac{zf''(z)}{f'(z)}\Big|,\;z\in \Delta.
\]

\par Also, it is clear that
$$f\in UCV \Leftrightarrow zf'(z)\in S_{p}.$$

Kanas and Wisniowska \cite{ch348,ch349} introduced $k$-uniformly starlike functions and $k$-uniformly convex functions as follows.
\begin{align}
k-ST=\Big\{f:\;f\in S \text{ and } \Re\Big(\dfrac{zf'(z)}{f(z)}\Big) >k\Big|\dfrac{zf'(z)}{f(z)}-1\Big|,\;z\in \Delta,\;k\geq0 \Big\}\nonumber\\
k-UCV=\Big\{f:\;f\in S \text{ and } \Re\Big(1+\dfrac{zf'(z)}{f(z)}\Big) >k\Big|\dfrac{zf''(z)}{f'(z)}\Big|,\;z\in \Delta, \;k\geq0 \Big\}.\nonumber
\end{align}
Bharati, et al. \cite{ch345}, defined $k-ST(\beta)$ and $k-UCV(\beta)$ as follows.
A function $f\in \mathfrak{\mathcal{A}}$ is said to be in the class $k-ST(\beta)$  if and only if
\begin{align}
\Re\Big(\dfrac{zf'(z)}{f(z)}\Big)-\beta >k\Big|\dfrac{zf'(z)}{f(z)}-1\Big|,\;z\in \Delta.\label{3.1}
\end{align}
A function $f\in \mathfrak{\mathcal{A}}$ is said to be in the class $k-UCV(\beta)$ if and only if
\begin{align}
\Re\Big(1+\dfrac{zf''(z)}{f'(z)}\Big)-\beta >k\Big|\dfrac{zf''(z)}{f'(z)}\Big|,\;z\in \Delta.\label{3.2}
\end{align}
Sim et al.\cite{sim}, generalized above classes and introduced $k-ST(\alpha,\;\beta)$ and $k-UCV(\alpha,\;\beta)$ as below.
\begin{definition}
	A function $f\in \mathfrak{\mathcal{A}}$ is said to be in the class $k-ST(\alpha,\;\beta)$  if and only if
	\begin{align}
	\Re\Big\{\dfrac{zf'(z)}{f(z)}\Big\}-\beta >k\Big|\dfrac{zf'(z)}{f(z)}-\alpha\Big|,\;z\in \Delta,\label{3.3}
	\end{align} where $0\leq\beta<\alpha\leq1$ and $k(1-\alpha)<1-\beta$.
\end{definition}
 \begin{definition}
 	A function $f\in \mathfrak{\mathcal{A}}$ is said to be in the class $k-UCV(\alpha,\; \beta)$if and only if
 	\begin{align}
 	\Re\Big(1+\dfrac{zf''(z)}{f'(z)}\Big)-\beta >k\Big|1+\dfrac{zf''(z)}{f'(z)}-\alpha\Big|,\;z\in \Delta.\label{3.4}
 	\end{align}
 	where $0\leq\beta<\alpha\leq1$ and $k(1-\alpha)<1-\beta$.
 \end{definition}
\par In particular, for $\alpha=1,\;\beta=0$ the classes $k-ST(\alpha,\; \beta)$ and $k-UCV(\alpha,\; \beta)$ reduces to $k-ST$ and $k-UCV$ respectively. Further, for $\alpha=1$ these classes coincides with the classes studied by Nishiwaki et al. \cite{ch346} and Shams et al. \cite{ch347}. In 2017, Annamalai et al. \cite{ch351}, obtained second Hankel determinant of analytic functions involving conic domains.
\
\\
\
\textbf{Geometric Interpretation:~} A function $f\in k-ST(\alpha,\; \beta) $ and $k-UCV(\alpha,\; \beta)$ if and only if $\dfrac{zf'(z)}{f(z)}$ and $1+\dfrac{zf''(z)}{f'(z)}$, respectively takes all the values in the conic domain $\Omega_{k,\;\alpha,\;\beta}$.
\begin{align}
\Omega_{k,\;\alpha,\;\beta}=\big\{\omega:\;\omega\in \mathbb{C}\text{ and } k|\omega-\alpha|<\Re(\omega)-\beta\big\}\nonumber
\end{align}
or
\begin{align}
\Omega_{k,\;\alpha,\;\beta}=\big\{\omega:\;\omega\in \mathbb{C}\text{ and } k\sqrt{[\Re(\omega)-\alpha]^{2}+[\Im(\omega)]^{2}}<\Re(\omega)-\beta\big\},\nonumber
\end{align}
where $0\leq\beta<\alpha\leq1$ and $k(1-\alpha)<1-\beta$. Clearly $1\in \Omega_{k,\;\alpha,\;\beta}$
 and $\Omega_{k,\;\alpha,\;\beta}$ is bounded by the curve
\begin{align}
\partial\Omega_{k,\;\alpha,\;\beta}=\big\{\omega:\;\omega=u+iv \text{ and }k^{2}(u-\alpha)^{2}+k^{2}v^{2}=(u-\beta)^{2}\big\}.\nonumber
\end{align}
\begin{definition}
	The Caratheodory functions $p\in P$ is said to be in the class $\mathcal{P}(p_{k,\;\alpha,\;\beta})$ if and only if $p$ takes all the values in the conic domain $\Omega_{k,\;\alpha,\;\beta}$. Analytically it is defined as follows:
	\begin{align}
	\mathcal{P}(p_{k,\;\alpha,\;\beta})&=\{p:\;p\in \mathcal{P} \text{ and } p(\Delta)\subset \Omega_{k,\;\alpha,\;\beta}\},\nonumber\\
	\mathcal{P}(p_{k,\;\alpha,\;\beta})&=\{p:\;p\in \mathcal{P} \text{ and } p(z)\prec p_{k,\;\alpha,\;\beta},\;z\in\Delta\}.\nonumber
	\end{align}
\end{definition}
Note that  $\partial\Omega_{k,\;\alpha,\;\beta}$ represents conic section about real axis. In particular, $\Omega_{k,\;\alpha,\;\beta}$ represents an elliptic domain for $k>1$, parabolic domain for for  $k=1$, hyperbolic domain for $0<k<1.$ Sim et al. \cite{sim} obtained the functions $p_{k,\alpha\;\beta}(z)$ which play the role of extremal functions  of $\mathcal{P}(p_{k,\;\alpha,\;\beta})$ as
\begin{align}
p_{k,\alpha\;\beta}(z)=\left\{
                         \begin{array}{ll}
                           \dfrac{1+(1-2\beta)z}{1-z}, & \text{for\;} k=0; \\
                            \alpha+\dfrac{2(\alpha-\beta)}{\pi^{2}}\text{log}^{2}\Big(\dfrac{1+\sqrt{u_{k}(z)}}{1-\sqrt{u_{k}(z)}}\Big), & \text{for\;} k=1;\\
                           \dfrac{(\alpha-\beta)}{1-k^{2}}\text{cosh}\{{\mathfrak{u}(k)\text{log}\Big(\dfrac{1+\sqrt{u_{k}(z)}}{1-\sqrt{u_{k}(z)}}\Big)}\}+\dfrac{\beta-\alpha k^{2}}{1-k^{2}}, & \text{for\;} 0<k<1; \\
                           \dfrac{(\alpha-\beta)}{k^{2}-1}\text{sin}^{2}(\dfrac{\pi}{2K(k)}\int_{0}^{\omega}\dfrac{dt}{\sqrt{1-t^{2}}\sqrt{1-t^{2}k^{2}}})+\dfrac{\alpha k^{2}-\beta}{k^{2}-1}, & \text{for\;} k>1;
                         \end{array}
                       \right.\nonumber
\end{align}
where $\mathfrak{u}(k)=\dfrac{2}{\pi}\cos^{-1}k,$
 $\;u_{k}(z)=\dfrac{z+\rho_{k}}{1+\rho_{k}z}$ and
\begin{align}
\rho_{k}=\left\{
           \begin{array}{ll}
             \big(\dfrac{e^{A}-1}{e^{A}+1}\big)^{2}, & \text{for\;} k=1; \\
             \Big(\dfrac{\text{exp}\big(\dfrac{1}{u_{k}(z)}\text{arc cosh B}\big)-1}{\text{exp}\Big(\dfrac{1}{u_{k}(z)}\text{arc cosh}B\Big)+1}\Big)^{2}, & \text{for\;} 0<k<1; \\
             \sqrt{k}\text{sin}\Big[   \dfrac{2K(\kappa)}{\pi}\text{arc sin}C\Big], & \text{for\;} k>1;
           \end{array}\nonumber
         \right.
\end{align}
with $A=\sqrt{\dfrac{1-\alpha}{2(\alpha-\beta)}\pi},  \;B=\dfrac{1}{\alpha-\beta}(1-k^{2}-\beta+\alpha k^{2}),  \; C=\dfrac{1}{\alpha-\beta}(k^{2}-1+\beta-\alpha k^{2})$.\\
Also
\begin{align}
K(\kappa)&=\int_{0}^{\omega}\dfrac{dt}{\sqrt{1-t^{2}}\sqrt{1-t^{2}\kappa^{2}}}\;\;(0<\kappa<1),\nonumber\\
K'(\kappa)&=K(\sqrt{1-\kappa^{2}})\;\;(0<\kappa<1),\nonumber\\
\kappa &=\text{cosh}\Big(\dfrac{\pi K'(\kappa)}{4K(\kappa)}\Big).\nonumber
\end{align}
According to Koebe's  $\dfrac{1}{4}$ theorem, every analytic and univalent function $f$ in $\Delta$ has an inverse $ f^{-1}$ and is defined as
 \begin{equation}
f^{-1}(f(z)) = z, (z\in\Delta),  f(f^{-1}(w))=w\Big(|w|<r_{0}(f); r_{0}(f)\geq\dfrac{1}{4}\Big).\nonumber
\end{equation}
Also the function $f^{-1}$ can be written as
\begin{equation}
 f^{-1}(w) = w-a_{2}w^{2}+(2a_{2}^{2}-a_{3})w^{3}-(5a_{2}^{3}-5a_{2}a_{3}+a_{4})w^{4}+\ldots .\label{1.5}
 \end{equation}
 \begin{definition}
 	A function $ f\in \mathfrak{\mathcal{A}}$ is said to be bi-univalent if  both $f$ and analytic extension of $  f^{-1}$ in $\Delta$ are univalent in $\Delta$.  The class of all bi-univalent  functions is denoted by $\Sigma$.
 	That is a function $f$ is said to be bi-univalent if and only if
 	\begin{enumerate}
 		\item $f$ is an analytic and univalent function in $\Delta$.
 		\item There exists an analytic and univalent function $g$ in $\Delta$ such that\\ $f(g(z))=g(f(z))=z$ in $\Delta$.
 	\end{enumerate}
 \end{definition}
  The class of bi-univalent functions was introduced by Lewin \cite{ch13} in 1967.  Recently many researchers (\cite{SA-SY-2016b}, \cite{Ali-Ravi-Ma-Mina-class}, \cite{Jay-NM-JY}, \cite{HO-NM-VKB-Fekete}, \cite{HO-NM-JY-Hankel}, \cite{r5}, \cite{HMS-Caglar}, \cite{HMS-SSE-RMA:FILO-2015},
\cite{HMS-SG-FG:AMS-2016}, \cite{HMS-SG-FG:AM-2017}, \cite{HMS-AKM-PG}, \cite{HMS-GMS-NM-GJM}, \cite{hhms}, \cite{Zaprawa}) have  introduced and investigated several interesting subclasses of
the bi-univalent functions and they have found non-sharp estimates of  two Taylor-Maclaurin coefficients $|a_{2}|$, $|a_{3}|$, Fekete-Szeg$\ddot{o}$ inequality  and second Hankel determinants.
 In 2017, $\c{S}$ahsene Altinkaya, Sibel Yal\c{c}in \cite{sahs}, \cite{SA-SY-GJM-2017} estimated the coefficients  and Fekete-Szeg$\ddot{o}$ inequality for some subclasses of bi-univalent functions involving symmetric $q$-derivative operator subordinate to the generating function of Chebyshev polynomial.
 \begin{definition}
 	\cite{ch343}
 	Jackson defined $q-$derivative operator $D_{q}$ of an analytic function $f$ of the form \eqref{1.1}as follows:
 	\begin{align}
 	D_{q}f(z)=\left\{
 	\begin{array}{ll}
 	\dfrac{f(qz)-f(z)}{(q-1)z}, & for\; z\neq0, \\
 	f'(0), & for\;z=0
 	\end{array}
 	\right.\nonumber\\
 	D_{q}f(0)=f'(0)\; \text{and} \;D^{2}_{q}=D_{q}(D_{q}f(z)).\nonumber
 	\end{align}
 	If $f(z)=z^{n}$ for any positive integer $n$, the $q$-derivative of $f(z)$ is defined by
 	\begin{align}
 	D_{q}z^{n}=\dfrac{(qz)^{n}-z^{n}}{qz-z}=[n]_{q}z^{n-1},\nonumber
 	\end{align}
 	where $[n]_{q}=\dfrac{q^{n}-1}{q-1}$. As $q\rightarrow1^{-}$ and $k\in\mathbb{N}$, we have $[n]_{q}\rightarrow n$ and $\lim_{q\rightarrow 1}(D_{q}f(z))=f'(z)$ where $f'$ is normal derivative of $f$.
 	Therefore \begin{align}
 	D_{q}f(z)=1+ \sum \limits_{n=2}^{\infty}[n]_{q} a_{n}z^{n-1}.\nonumber
 	\end{align}
 \end{definition}
 \begin{definition}
 	\cite{ch344}
 	The symmetric $q-$derivative operator $\widetilde{D}_{q}$ of an analytic function $f$ is defined as follows:
 	\begin{align}
 	(\widetilde{D}_{q}f)(z)=\left\{
 	\begin{array}{ll}
 	\dfrac{f(qz)-f(q^{-1}z)}{(q-q^{-1})z}, & \text{for}\; z\neq0, \\
 	f'(0), & \text{for}\;z=0
 	\end{array}.
 	\right.\nonumber
 	\end{align}
 \end{definition}
 It is clear that $\widetilde{D}_{q}z^{n}=\widetilde{[n]}_{q}z^{n-1}$ and $\widetilde{D_{q}f}(z)=1+ \sum \limits_{n=2}^{\infty}\widetilde{[n]}_{q} a_{n}z^{n-1}$, where $\widetilde{[n]}_{q}=\dfrac{q^{n}-q^{-n}}{q-q^{-1}}$.
 The relation between $q$-derivative operator and symmetric $q$-derivative operator is given by
 \begin{align}
 (\widetilde{D_{q}}f)(z)=D_{q^{2}}f(q^{-1}z).\nonumber
 \end{align}
 If $g$ is the inverse of $f$ then
 \begin{align}
 (\widetilde{D_{q}}g)(w)&=\dfrac{g(qw)-g(q^{-1}w)}{(q-q^{-1})w}\nonumber\\
 &=1-\widetilde{[2]_{q}}a_{2}w+\widetilde{[3]_{q}}(2a_{2}^{2}-a_{3})w^{2}-\widetilde{[4]_{q}}(5a_{2}^{3}-5a_{2}a_{3}+a_{4})w^{3}+\ldots.\nonumber
 \end{align}
 \par The $q$-calculus has so many applications in various branches of mathematics and physics. Jackson   \cite{ch343} developed $q$-integral and $q$-derivative in a systematic way. The fractional $q$-calculus is an important tool used to study various families of analytic functions. 
 In recent years, several subclasses of analytic functions involving fractional $q$-integral and fractional $q$-derivative operators were constructed and coefficient inequality, Fekete-Szeg\"{o} inequality  and Hankel determinant were estimated for the functions in these classes.
\par Motivated by the above mentioned work, in this paper, bi-starlike functions of order $b$ and bi-convex functions of order $b$ involving $q$-derivative operator subordinate to the conic domains are defined and the Fekete-Szeg$\ddot{o}$ inequality for the function in these classes are obtained.
 \begin{definition}
 	A function $f\in\Sigma$ is said to be  in the class $k-ST_{\Sigma, \;b}(\alpha,\;\beta);\text{ where } 0\leq\beta<\alpha\leq1$ and $k(1-\alpha)<1-\beta$,  and $b$ is a non-zero complex number, if it satisfies the following conditions:
 	\begin{align}
 	1+\dfrac{1}{b}\Big[\dfrac{z \widetilde{D}_{q}f(z)}{f(z)}-1\Big]\prec p_{k,\;\alpha,\;\beta}(z) \qquad \hbox{and} \qquad 
 	1+\dfrac{1}{b}\Big[\dfrac{w \widetilde{D}_{q}g(w)}{g(w)}-1\Big]\prec p_{k,\;\alpha,\;\beta}(w)\label{3.6}
 	\end{align}
 	where $g$ is an extension of $f^{-1}$ to $\Delta$.
 \end{definition}
 \begin{definition}
 	A function $f\in\Sigma$ is said to be  in the class $k-UCV_{\Sigma, \;b}(\alpha,\;\beta); \text{ where } 0\leq\beta<\alpha\leq1$ and $k(1-\alpha)<1-\beta$, and $b$ is a non-zero complex number,  if it satisfies the following conditions:
 	\begin{align}
 	1+\dfrac{1}{b}\Big[\dfrac{z \widetilde{D}_{q}(\widetilde{D}_{q}f(z))}{\widetilde{D}_{q}(f(z))}\Big]\prec p_{k,\;\alpha,\;\beta}(z)\qquad \hbox{and}\qquad
 	1+\dfrac{1}{b}\Big[\dfrac{w \widetilde{D}_{q}(\widetilde{D}_{q}g(w))}{\widetilde{D}_{q}(g(w))}\Big]\prec p_{k,\;\alpha,\;\beta}(w)\label{3.8}
 	\end{align}
 	where $g$ is an extension of $f^{-1}$ to $\Delta$.
 \end{definition}
\section{Main Results}
In this section, Fekete-Szeg\"{o} inequality for the functions in the $f$  classes  $k-ST_{\Sigma, \;b}(\alpha,\;\beta) $ and $k-UCV_{\Sigma, \;b}(\alpha,\;\beta) $  are estimated.
\begin{theorem}\label{th1}
	If $f\in k-ST_{\Sigma, \;b}(\alpha,\;\beta) $ and is of the form \eqref{1.1} then
\begin{align}
|a_{2}|&\leq\dfrac{P_{1}\sqrt{P_{1}}b^{2}}{\sqrt{[P_{1}^{2}b(\widetilde{[3]}_{q}-\widetilde{[2]}_{q})+2(P_{1}-P_{2})(\widetilde{[2]}_{q}-1)^{2}]}},\;\;
|a_{3}|&\leq\dfrac{b^{2}P_{1}^{2}}{(\widetilde{[2]}_{q}-1)^{2}}+\dfrac{bP_{1}}{(\widetilde{[3]}_{q}-1)}\nonumber
\end{align}
and
\begin{align}
|a_{3}-\mu a_{2}^{2}|&\leq\left\{
                           \begin{array}{ll}
                           \dfrac{P_{1}b}{(\widetilde{[3]}_{q}-1)}, & \text{if } 0\leq|s(\mu)|\leq1 \\
                              \dfrac{P_{1}b|s(\mu)|}{(\widetilde{[3]}_{q}-1)}& \text{if } |s(\mu)|\geq1,
                           \end{array}
                         \right.\nonumber
\end{align}
where
\begin{align}
s(\mu)=\dfrac{P_{1}^{2}b(1-\mu)}{4[P_{1}^{2}b(\widetilde{[3]}_{q}-\widetilde{[2]}_{q})+2(P_{1}-P_{2})(\widetilde{[2]}_{q}-1)^{2}]}.\nonumber
\end{align}
\end{theorem}
\begin{proof}
	Let $f\in k-ST_{\Sigma, \;b}(\alpha,\;\beta)$  and $g$ be  an analytic extension of $f^{-1}$ in $\Delta$. Then  there exist two Schwarz  functions $u, v $ in $\Delta $ such that
	\begin{align}
	1+\dfrac{1}{b}\Big[\dfrac{z \widetilde{D}_{q}f(z)}{f(z)}-1\Big]=P_{k,\;\alpha,\;\beta}(u(z)),\label{3.1}\\
	1+\dfrac{1}{b}\Big[\dfrac{w \widetilde{D}_{q}g(w)}{g(w)}-1\Big]=P_{k,\;\alpha,\;\beta}(v(w)).\label{3.2}
	\end{align}
	Define two functions $h$, $q \in P$  such that
	\begin{equation}
	h(z)=\dfrac{1+u(z)}{1-u(z)}=1+h_{1}z+h_{2}z^{2}+h_{3}z^{3}+\ldots\nonumber
	\end{equation}
	and
	\begin{equation}
	q(w)=\dfrac{1+v(w)}{1-v(w)}=1+q_{1}w+q_{2}w^{2}+q_{3}w^{3}+\ldots .\nonumber
	\end{equation}
	Then
	\begin{align}
	 P_{k,\;\alpha,\;\beta}\left(\dfrac{h(z)-1}{h(z)+1}\right)=&1+\dfrac{P_{1}h_{1}z}{2}+\left(\dfrac{P_{1}}{2}(h_{2}-\dfrac{h_{1}^{2}}{2})+\dfrac{P_{2}h_{1}^{2}}{4}\right)z^{2}\nonumber\\+&
	\left(\dfrac{P_{1}}{2}\left(\dfrac{h_{1}^{3}}{4}-h_{1}h_{2}+h_{3}\right)
	+\dfrac{P_{2}}{4}(2h_{1}h_{2}-h_{1}^{3})+\dfrac{P_{3}}{8}h_{1}^{3}\right)z^{3}+\ldots\nonumber\\
	 P_{k,\;\alpha,\;\beta}\left(\dfrac{v(w)-1}{v(w)+1}\right)=&1+\dfrac{P_{1}q_{1}w}{2}+\left(\dfrac{P_{1}}{2}(q_{2}-\dfrac{q_{1}^{2}}{2})+\dfrac{P_{2}q_{1}^{2}}{4}\right)w^{2}\nonumber\\+&
	\left(\dfrac{P_{1}}{2}\left(\dfrac{q_{1}^{3}}{4}-q_{1}q_{2}+q_{3}\right)
	+\dfrac{P_{2}}{4}(2q_{1}q_{2}-q_{1}^{3})+\dfrac{P_{3}}{8}q_{1}^{3}\right)w^{3}+\ldots .\nonumber
	\end{align}
	Then the above equations become                                                                                                                                                                           \begin{align}
	1+\dfrac{1}{b}\Big[\dfrac{z \widetilde{D}_{q}f(z)}{f(z)}-1\Big]=P_{k,\;\alpha,\;\beta}\left(\dfrac{h(z)-1}{h(z)+1}\right),\label{3.3}\\
	1+\dfrac{1}{b}\Big[\dfrac{w \widetilde{D}_{q}g(w)}{g(w)}-1\Big]=P_{k,\;\alpha,\;\beta}\left(\dfrac{v(w)-1}{v(w)+1}\right).\label{3.4}
	\end{align}
	Comparing the coefficients of similar powers of $z$ in equations \eqref{3.5} and \eqref{3.6}, we get
	\begin{align}
	\dfrac{1}{b}(\widetilde{[2]}_{q}-1)a_{2}&=\dfrac{P_{1}h_{1}}{2},\label{3.5}\\
	 \dfrac{1}{b}[(\widetilde{[3]}_{q}-1)a_{3}-(\widetilde{[2]}_{q}-1)a_{2}^{2}]&=\dfrac{P_{1}}{2}(h_{2}-\dfrac{h_{1}^{2}}{2})+\dfrac{P_{2}h_{1}^{2}}{4},\label{3.6}
	\end{align}
	and
	\begin{align}
	\dfrac{-1}{b}(\widetilde{[2]}_{q}-1)a_{2}&=\dfrac{P_{1}q_{1}}{2},\label{3.7}\\
	 \dfrac{1}{b}[(\widetilde{[3]}_{q}-1)(2a_{2}^{2}-a_{3})-(\widetilde{[2]}_{q}-1)a_{2}^{2}]&=\dfrac{P_{1}}{2}(q_{2}-\dfrac{q_{1}^{2}}{2})+\dfrac{P_{2}q_{1}^{2}}{4}.\label{3.8}
	\end{align}
	From the equations  \eqref{3.5} and \eqref{3.7}
	\begin{align}
	h_{1}=-q_{1}.\label{3.9}
	\end{align}
	Now squaring and adding the equations \eqref{3.5} from \eqref{3.7}, we get
	\begin{align}
	h_{1}^{2}+q_{1}^{2}=\dfrac{8(\widetilde{[2]}_{q}-1)^{2}a_{2}^{2}}{P_{1}^{2}b^{2}}.\label{3.10}
	\end{align}
	Now adding \eqref{3.6} and \eqref{3.8}, use the equation \eqref{3.10}, one can get
	\begin{align}
	 a_{2}^{2}=\dfrac{P_{1}^{3}(h_{2}+q_{2})b^{2}}{4[P_{1}^{2}b(\widetilde{[3]}_{q}-\widetilde{[2]}_{q})+2(P_{1}-P_{2})(\widetilde{[2]}_{q}-1)^{2}]}.\label{3.11}
	\end{align}
	Now subtract  the equation \eqref{3.8}  from \eqref{3.6},
	\begin{align}
	a_{3}= a_{2}^{2}+\dfrac{bP_{1}(h_{2}-q_{2})}{4(\widetilde{[3]}_{q}-1)}. \label{3.12}
	\end{align}
	Then using the equation \eqref{3.10}, we get
	\begin{align}
	a_{3}= \dfrac{P_{1}^{3}b^{2}(h_{1}^{2}+q_{1}^{2})}{8(\widetilde{[2]}_{q}-1)^{2}}+\dfrac{bP_{1}(h_{2}-q_{2})}{4(\widetilde{[3]}_{q}-1)}.\label{3.13}
	\end{align}
	Then using the equations \eqref{3.11} and \eqref{3.12}, we get
	\begin{align}
	a_{3}-\mu  a_{2}^{2}=\dfrac{bP_{1}}{4(\widetilde{[3]}_{q}-1)}\Big[h_{2}(1+s(\mu))+q_{2}(-1+s(\mu))\Big],\label{3.14}
	\end{align}
	where
	\begin{align}
	s(\mu)=\dfrac{P_{1}^{2}b(1-\mu)}{4[P_{1}^{2}b(\widetilde{[3]}_{q}-\widetilde{[2]}_{q})+2(P_{1}-P_{2})(\widetilde{[2]}_{q}-1)^{2}]}~.\nonumber
	\end{align}
	By applying the modulus for the equations \eqref{3.11}, \eqref{3.13} and \eqref{3.14}, we get the required results.
\end{proof}
\begin{theorem}\label{th2}
	If $f\in k-UCV_{\Sigma, \;b}(\alpha,\;\beta) $ and is of the form \eqref{1.1} then
	\begin{align}
	 |a_{2}|&\leq\dfrac{P_{1}\sqrt{P_{1}b}}{\sqrt{\Big[2\widetilde{[2]}_{q}(\widetilde{[3]}_{q}-\widetilde{[2]}_{q})bP_{1}^{2}+\widetilde{[2]}_{q}^{2}(P_{1}-P_{2})\Big]}}\; \hbox{and}\;
	|a_{3}|\leq\dfrac{P_{1}^{2}b^{2}}{\widetilde{[2]}_{q}^{2}}+\dfrac{bP_{1}}{\widetilde{[2]}_{q}\widetilde{[3]}_{q}}\nonumber
	\end{align}
	and
	\begin{align}
	|a_{3}-\mu a_{2}^{2}|&\leq\left\{
	\begin{array}{ll}
	\dfrac{P_{1}b}{\widetilde{[2]}_{q}\widetilde{[3]}_{q}}, & \text{if } 0\leq|s(\mu)|\leq 1 \\
	\dfrac{P_{1}b|s(\mu)|}{\widetilde{[2]}_{q}\widetilde{[3]}_{q}}& \text{if } |s(\mu)|\geq 1,
	\end{array}
	\right.\nonumber
	\end{align}
	where
	\begin{align}
	 s(\mu)=\dfrac{P_{1}^{2}b(1-\mu)}{4\Big[2\widetilde{[2]}_{q}(\widetilde{[3]}_{q}-\widetilde{[2]}_{q})bP_{1}^{2}+\widetilde{[2]}_{q}^{2}(P_{1}-P_{2})\Big]}~.\nonumber
	\end{align}
\end{theorem}
\begin{proof}
	If $f\in k-UCV_{\Sigma, \;b}(\alpha,\;\beta)$  and $g$ is an analytic extension of $f^{-1}$ in $\Delta$, then there exist two Schwarz  functions $u, v $ in $\Delta $ such that
	\begin{align}
	1+\dfrac{1}{b}\Big[\dfrac{z \widetilde{D}_{q}(\widetilde{D}_{q}f(z))}{\widetilde{D}_{q}(f(z))}\Big]=p_{k,\;\alpha,\;\beta}(u(z)),\label{3.15}\\
	1+\dfrac{1}{b}\Big[\dfrac{w \widetilde{D}_{q}(\widetilde{D}_{q}g(w))}{\widetilde{D}_{q}(g(w))}\Big]=p_{k,\;\alpha,\;\beta}(v(w)).\label{3.16}
	\end{align}
	Define two functions $h$, $q$ such that
	\begin{equation}
	h(z)=\dfrac{1+u(z)}{1-u(z)}=1+h_{1}z+h_{2}z^{2}+h_{3}z^{3}+\ldots\nonumber
	\end{equation}
	and
	\begin{equation}
	q(w)=\dfrac{1+v(w)}{1-v(w)}=1+q_{1}w+q_{2}w^{2}+q_{3}w^{3}+\ldots.\nonumber
	\end{equation}
	Then
	\begin{align}
	 P_{k,\;\alpha,\;\beta}\left(\dfrac{h(z)-1}{h(z)+1}\right)=&1+\dfrac{P_{1}h_{1}z}{2}+\left(\dfrac{P_{1}}{2}(h_{2}-\dfrac{h_{1}^{2}}{2})+\dfrac{P_{2}h_{1}^{2}}{4}\right)z^{2}\nonumber\\+&
	\left(\dfrac{P_{1}}{2}\left(\dfrac{h_{1}^{3}}{4}-h_{1}h_{2}+h_{3}\right)
	+\dfrac{P_{2}}{4}(2h_{1}h_{2}-h_{1}^{3})+\dfrac{P_{3}}{8}h_{1}^{3}\right)z^{3}+\ldots\nonumber\\
	 P_{k,\;\alpha,\;\beta}\left(\dfrac{v(w)-1}{v(w)+1}\right)=&1+\dfrac{P_{1}q_{1}w}{2}+\left(\dfrac{P_{1}}{2}(q_{2}-\dfrac{q_{1}^{2}}{2})+\dfrac{P_{2}q_{1}^{2}}{4}\right)w^{2}\nonumber\\+&
	\left(\dfrac{P_{1}}{2}\left(\dfrac{q_{1}^{3}}{4}-q_{1}q_{2}+q_{3}\right)
	+\dfrac{P_{2}}{4}(2q_{1}q_{2}-q_{1}^{3})+\dfrac{P_{3}}{8}q_{1}^{3}\right)w^{3}+\ldots\nonumber
	\end{align}
	Then the above equations  reduces to
	\begin{align}
	1+\dfrac{1}{b}\Big[\dfrac{z \widetilde{D}_{q}(\widetilde{D}_{q}f(z))}{\widetilde{D}_{q}(f(z))}\Big]=&P_{k,\;\alpha,\;\beta}\left(\dfrac{h(z)-1}{h(z)+1}\right),\label{3.17}\\
	1+\dfrac{1}{b}\Big[\dfrac{w \widetilde{D}_{q}(\widetilde{D}_{q}g(w))}{\widetilde{D}_{q}(g(w))}\Big]=&P_{k,\;\alpha,\;\beta}\left(\dfrac{v(w)-1}{v(w)+1}\right).\label{3.18}
	\end{align}
	Comparing the coefficients of similar powers of $z$ in equations \eqref{3.17} and \eqref{3.18}
	\begin{align}
	\dfrac{1}{b}\widetilde{[2]}_{q}a_{2}&=\dfrac{P_{1}h_{1}}{2},\label{3.19}\\
	 \dfrac{\widetilde{[2]}_{q}\widetilde{[3]}_{q}a_{3}-\widetilde{[2]}_{q}^{2}a_{2}^{2}}{b}&=\dfrac{P_{1}}{2}\big(h_{2}-\dfrac{h_{1}^{2}}{2}\big)+\dfrac{P_{2}h_{1}^{2}}{4},\label{3.20}
	\end{align}
	and
	\begin{align}
	\dfrac{-1}{b}\widetilde{[2]}_{q}a_{2}&=\dfrac{P_{1}q_{1}}{2}, \label{3.21}\\
	 \dfrac{1}{b}(\widetilde{[2]}_{q}\widetilde{[3]}_{q}(2a_{2}^{2}-a_{3})-\widetilde{[2]}_{q}^{2}a_{2}^{2})&=\dfrac{P_{1}}{2}\big(q_{2}-\dfrac{q_{1}^{2}}{2}\big)+\dfrac{P_{2}q_{1}^{2}}{4}.\label{3.22}
	\end{align}
	From the equations  \eqref{3.19} and \eqref{3.21}, we get
	\begin{align}
	h_{1}=-q_{1}.\label{3.23}
	\end{align}
	Squaring and adding the equations \eqref{3.19} from \eqref{3.21}, we get
	\begin{align}
	h_{1}^{2}+q_{1}^{2}=\dfrac{4(\widetilde{[2]}_{q})^{2}a_{2}^{2}}{P_{1}^{2}b^{2}}.\label{3.24}
	\end{align}
	Adding \eqref{3.20} and \eqref{3.22}, and using the equation \eqref{3.24}, one can get
	\begin{align}
	 a_{2}^{2}=\dfrac{P_{1}^{3}(h_{2}+q_{2})b^{2}}{4[2\widetilde{[2]}_{q}(\widetilde{[3]}_{q}-\widetilde{[2]}_{q})P_{1}^{2}+(\widetilde{[2]}_{q})^{2}(P_{1}-P_{2})]}.\label{3.25}
	\end{align}
	Subtracting  the equation \eqref{3.22}  from \eqref{3.20}, we get
	\begin{align}
	a_{3}= a_{2}^{2}+\dfrac{bP_{1}(h_{2}-q_{2})}{4(\widetilde{[2]}_{q}\widetilde{[3]}_{q})}. \label{3.26}
	\end{align}
	Using the equation \eqref{3.24}, we obtain
	\begin{align}
	a_{3}=\dfrac{P_{1}^{2}b^{2}(h_{1}^{2}+q_{1}^{2})}{8\widetilde{[2]}_{q}^{2}}+\dfrac{bP_{1}(h_{2}-q_{2})}{4(\widetilde{[2]}_{q}\widetilde{[3]}_{q})}. \label{3.27}
	\end{align}
	Then using the equations \eqref{3.25} and \eqref{3.26}, we get
	\begin{align}
	a_{3}-\mu  a_{2}^{2}=\dfrac{bP_{1}}{4(\widetilde{[2]}_{q}\widetilde{[3]}_{q})}\Big[h_{2}(1+s(\mu))+q_{2}(-1+s(\mu))\Big],\label{3.28}
	\end{align}
	where
	\begin{align}
	 s(\mu)=\dfrac{bP_{1}^{2}(1-\mu)}{4[2\widetilde{[2]}_{q}(\widetilde{[3]}_{q}-\widetilde{[2]}_{q})bP_{1}^{2}+\widetilde{[2]}_{q}^{2}(P_{1}-P_{2})]2}.\nonumber
	\end{align}
	By applying modulus for the equations   \eqref{3.25},  \eqref{3.27}  and \eqref{3.28} on both sides we get the required results.
\end{proof}
\
\textbf{Acknowledgement:}The work presented in this paper is partially supported by DST-FIST-Grant No.SR/FST/MSI-101/2014, dated 14/1/2016.

\end{document}